\documentclass[a4paper,11pt]{article}
\usepackage{amsfonts, amsmath, amssymb, amsthm, verbatim}
\newtheorem{theorem}{Theorem}
\newtheorem{proposition}{Proposition}
\newtheorem{lemma}{Lemma}
\newtheorem{corollary}{Corollary}

\theoremstyle{definition}

\newtheorem{remark}{Remark}
\newtheorem{assump}{Assumption}

\newcommand{\R}{\mathbb R}
\newcommand{\N}{\mathbb N}
\newcommand{\OO}{\mathcal O}
\newcommand{\F}{\mathcal F}
\newcommand{\C}{\mathcal C}
\newcommand{\f}{{\tilde f}}
\newcommand{\rhot}{{\tilde \rho}}

\newcommand{\doo}{\overline{d}}

\newcommand{\dd}{\mathbf{d}}
\newcommand{\U}{V}
\newcommand{\bbb}{\mathbf{b}}

\DeclareMathOperator{\Ric}{Ric}
\DeclareMathOperator{\Sect}{Sect}
\DeclareMathOperator{\Cut}{Cut}
\DeclareMathOperator{\tr}{tr}
\DeclareMathOperator{\vol}{vol}
\newcommand{\Cutst}{\Cut_{\mathrm{ST}}}

\author{
  Kazumasa Kuwada\footnote{Partially supported by the JSPS fellowship for research abroad}
  \,
  and Robert Philipowski
}
\title{Non-explosion of diffusion processes on manifolds with time-dependent metric}

\date{}
\begin{document}

\maketitle

\begin{abstract}
We study the problem of non-explosion of diffusion processes on a manifold with time-dependent Riemannian metric.
In particular we obtain that Brownian motion cannot explode in finite time if the metric evolves under backwards Ricci flow.
Our result makes it possible to remove the assumption of non-explosion in the pathwise contraction result
established by Arnaudon, Coulibaly and Thalmaier (arXiv:0904.2762, to appear in S\'em. Prob.).

As an important tool which is of independent interest we derive an It\^o formula for the distance from a fixed reference point,
generalising a result of Kendall (Ann.~Prob.~15 (1987), 1491--1500).
\end{abstract}

\section{Brownian motion with respect to time-changing Riemannian metrics} 
Let $M$ be a $d$-dimensional differentiable manifold, $\pi: \F(M) \to M$ the frame bundle and $(g(t))_{t \in [0,T]}$
a family of Riemannian metrics on $M$ depending smoothly on $t$ such that $(M, g(t))$ is geodesically complete for all $t \in [0,T]$.
Let $(e_i)_{i = 1}^d$ be the standard basis of $\R^d$.
For each $t \in [0,T]$ let $(H_i(t))_{i=1}^d$ be the associated $g(t)$-horizontal vector fields on $\F(M)$
(i.e.~$H_i(t,u)$ is the $g(t)$-horizontal lift of $ue_i$), and let $(V_{\alpha, \beta})_{\alpha, \beta = 1}^d$
be the canonical vertical vector fields. 
Let $(W_t)_{t \geq 0}$ be a standard $\R^d$-valued Brownian motion.
In this situation Arnaudon, Coulibaly and Thalmaier~\cite{act, coulibaly} defined horizontal Brownian motion
on $\F(M)$ as the solution of the following Stratonovich SDE:
\[
dU_t = \sum_{i=1}^d H_i(t, U_t) \circ dW_t^i - \frac{1}{2} \sum_{\alpha, \beta = 1}^d \frac{\partial g}{\partial t}(t, U_te_\alpha, U_t e_\beta) V_{\alpha \beta}(U_t) dt.
\]
They showed that if $U_0 \in \OO_{g(0)}(M)$, then $U_t \in \OO_{g(t)}(M)$ for all $t \in [0,T]$.
$g(t)$-Brownian motion on $M$ is then defined as $X_t := \pi U_t$.
We denote the law of $g(t)$-Brownian motion on $M$ started at $x$ by $P^x$, and expectation with respect to that measure by $E^x$.

\section{Main result}

The main result of this paper is the following theorem:

\begin{theorem}\label{mainresult}
If the family of metrics evolves under backwards super Ricci flow, i.e.
\begin{equation} \label{backwards}
\frac{\partial g}{\partial t} \leq \Ric,
\end{equation}
then Brownian motion on $M$ cannot explode up to time $T$. 
In particular this result holds for backwards Ricci flow $\frac{\partial g}{\partial t} = \Ric$.
\end{theorem}

By recent work (see section~\ref{related}),
it has turned out that backwards Ricci flow tends to compensate the effects of Ricci curvature on the behaviour of heat flow and Brownian motion.
Thus our result is quite natural 
because a lower bound of Ricci curvature yields the non-explosion property in the fixed metric case. 

\begin{remark}\hfill 
\begin{enumerate}
\item
In Section~\ref{nonsymmetric} we will give an extension of Theorem~\ref{mainresult} including the case of non-symmetric diffusion processes.
\item
For the question of explosion or non-explosion of Brownian motion on a manifold equipped with a fixed Riemannian metric
see e.g.~\cite{grigoryan}, \cite[Section~7.8]{hackenbrochthalmaier} or \cite[Section~4.2]{hsu}.
\end{enumerate}
\end{remark}

As an important tool we prove the following It\^o formula for the radial process $\rho(t, X_t)$,
where $\rho(t,x)$ denotes the distance with respect to $g(t)$ between $x$ and a fixed reference point $o$:

\begin{theorem}\label{itoradial}
There exists a non-decreasing continuous process $L$ which increases only when $X_t \in \Cut_{g(t)}(o)$ such that
\begin{equation} \label{eq:Itoradial} 
\rho(t, X_t) = \rho(0, X_0) + \int_0^t \left[ \frac{1}{2} \Delta_{g(s)} \rho + \frac{\partial \rho}{\partial s} \right] (s, X_s) ds
+ \sum_{i = 1}^d \int_0^t (U_s e_i) \rho(s,X_s) d W_s^i - L_t.
\end{equation}
\end{theorem}

\begin{remark} \hfill
\begin{enumerate}
\item The usual It\^o formula fails to apply because the distance function is not smooth at the cut-locus.
A priori it is even not clear that $\rho(t,X_t)$ is a semimartingale.

\item In the case of a fixed Riemannian metric Theorem~\ref{itoradial} was proved by Kendall~\cite{kendall}
(see also \cite[Theorem~7.254]{hackenbrochthalmaier} or \cite[Theorem~3.5.1]{hsu}).
The idea of our proof is based on Kendall's original one. 
\end{enumerate}

\end{remark}

\section{Remarks concerning related work} \label{related}
McCann and Topping~\cite{mccanntopping} 
(see Topping~\cite{topping} and Lott~\cite{lott} also) 
showed contraction in the Wasserstein metric 
for the heat equation under backwards Ricci flow on a compact manifold. 
More precisely, they showed that the following are equivalent:
\begin{enumerate}
\item \label{super}
$g$ evolves under backwards super Ricci flow, i.e.~$\frac{\partial g}{\partial t} \leq \Ric$.

\item \label{contraction}
Whenever $u$ and $v$ are two non-negative unit-mass solutions of the heat equation
\[
\frac{\partial u}{\partial t} = \frac{1}{2} \Delta_{g(t)} u - \left( \frac{1}{2} \tr \frac{\partial g}{\partial t} \right) u
\]
(the term $\left( \frac{1}{2} \tr \frac{\partial g}{\partial t} \right) u$ comes from the change in time of the volume element),
the function $t \mapsto W_2(t, u(t, \cdot) \vol_{g(t)}, v(t, \cdot) \vol_{g(t)})$ is non-increasing. Here
\[
W_2(t, \mu, \nu) := \left( \inf_{\pi} \int_{M \times M} d_{g(t)} (x,y)^2 \pi(dx, dy) \right)^{1/2}
\]
is the $L^2$-Wasserstein distance of two probability measures $\mu$ and $\nu$ on $M$. (The infimum is over all probability measures $\pi$
on $M \times M$ whose marginals are $\mu$ and $\nu$.)
\end{enumerate}
It means that backwards super Ricci flow is characterised by the contractivity property for solutions of the heat equation. 
%
%
Moreover, in recent work by 
Topping~\cite{topping} and Lott~\cite{lott} 
(see  Brendle~\cite{brendle} also) 
the heat equation and the theory of optimal transport 
are efficiently used to 
derive several monotonicity results 
including a new proof for the monotonicity of 
Perelman's reduced volume. 
These facts indicate that 
it would be effective 
for deeper understanding of Ricci flow 
to study the heat equation 
in conjunction with backwards Ricci flow 
and the theory of optimal transport. 

The non-explosion property of the Brownian motion 
is one of the first problems we face 
when we begin to consider 
the heat equation on a noncompact manifold. 
Our result tells us that 
it is always satisfied 
as far as we consider 
the heat equation under backwards Ricci flow. 
It will be quite helpful 
for the study of Ricci flow 
on a noncompact manifold 
by means of the heat equation. 
In fact, 
our result enables us to remove the assumption 
on the non-explosion in recent work 
by Arnaudon, Coulibaly and Thalmaier~\cite[Section~4]{act2}. 
They extend McCann and Topping's implication 
\ref{super} $\Rightarrow$ \ref{contraction} 
in the case on a noncompact manifold. 
In addition, 
they sharpen the monotonicity of 
$L^2$-Wasserstein distance 
to a pathwise contraction 
in the following sense; 
There is a coupling 
$( \bar{X}_t^{(1)} , \bar{X}^{(2)}_t)_{t \ge 0}$ of 
two Brownian motions 
starting from $x,y \in X$ respectively 
so that 
$t \mapsto d_{g(t)} ( \bar{X}^{(1)}_t , \bar{X}^{(2)}_t )$ 
is non-increasing almost surely. 
By taking an expectation, 
we can derive the monotonicity of 
the $L^2$-Wasserstein distance from it. 
The sharpness of their pathwise contraction 
looks 
useful for the study of 
the optimal transport associated with 
a more general cost function 
than the squared distance, 
e.g.~$\mathcal{L}$-optimal transportation 
studied in the above mentioned papers 
\cite{topping,lott,brendle}. 
As a consequence of our result, 
we can consider such a problem 
without assuming the compactness of the underlying space. 

\section{Proof of Theorem~\ref{itoradial}: It\^o's formula for the radial process}
Since it suffices to prove Theorem~\ref{itoradial} before the exit time of $X$ of an arbitrarily large relatively compact open subset of $M$,
we may assume that $M$ is compact and that therefore its injectivity radius 
\[
i_M := \inf \{ d_{g(t)}(x, y) \, | \, t \in [0,T], y \in \Cut_{g(t)}(x) \}
\]
is strictly positive and that we have a uniform bound for the sectional curvature $\Sect_{g(t)}$:
\[
|\Sect_{g(t)} | \leq K^2 \quad \mbox{for all } t \in [0,T].
\]
We first state It\^o's formula for smooth functions:

\begin{lemma}\label{ito2}
Let $f$ be a smooth function on $[0,T] \times M$. Then
\begin{eqnarray*}
d f(t, X_t) 
& = & \frac{\partial f}{\partial t}(t, X_t) dt + \frac{1}{2} \Delta_{g(t)} f(t,X_t) dt + \sum_{i = 1}^d (U_t e_i) f(t,X_t) d W_t^i.
\end{eqnarray*}
\end{lemma}

\begin{proof}
It\^o's formula applied to a smooth function $\f$ on $[0,T] \times \F(M)$ gives
\begin{eqnarray}
d \f(t, U_t) & = & \frac{\partial \f}{\partial t}(t, U_t) dt + \sum_{i = 1}^d H_i(t) \f(t,U_t) d W_t^i + \frac{1}{2} \sum_{i=1}^d H_i(t)^2 \f(t,U_t) dt \nonumber\\
&& {} - \frac{1}{2} \sum_{\alpha, \beta = 1}^d \frac{\partial g}{\partial t}(t, U_t e_\alpha, U_t e_\beta) V_{\alpha \beta} \f (t, U_t) dt. \label{ito1}
\end{eqnarray}
Now let $\f(t,u) := f(t,\pi u)$. By definition of $H_i(t)$, $H_i(t) \f(t,u) = (u e_i) f(t,\pi u)$.
Moreover, it is well known (see e.g.~\cite[Proposition~3.1.2]{hsu}) that $\sum_{i=1}^d H_i(t)^2 \f(t,u) = \Delta_{g(t)} f(t,\pi u)$.
Finally, since $\f$ is constant in the vertical direction, the last term in \eqref{ito1} vanishes, so that the claim follows.
\end{proof}

\begin{lemma}\label{density}
Let $G(x, \tau, y, t)$ ($x,y \in M$, $0 \leq t < \tau \leq T$) be the fundamental solution of the equation
$\frac{\partial u}{\partial t} = \frac{1}{2} \Delta_{g(t)}u$ (see \cite{guenther} for existence).
Then for all $\tau \in (0,T]$ and all $x \in M$ the law of $X_\tau$ under $P^x$ is absolutely continuous with respect to the volume measure
(note that this property does not depend on the choice of the Riemannian metric), and its density with respect to the 
$g(0)$-volume measure is given by $y \mapsto G(x, \tau, y, 0)$.
\end{lemma}

\begin{proof}
Fix $\varphi \in \C^2(M)$, and let $u$ be the solution of the initial value problem
\[
\left\{ \begin{array}{lll}
\frac{\partial u}{\partial t} & = & \frac{1}{2} \Delta_{g(t)}u\\
u(0, \cdot) & = & \varphi.
\end{array} \right.
\]
Then by Corollary~2.2 in \cite{guenther},
\[
u(\tau, x) = \int_M G(x, \tau, y, 0) \varphi(y) d vol_{g(0)}(y).
\]
Now apply It\^o's formula to $X$ and the function $(t,x) \mapsto u(\tau-t,x)$ to obtain
\begin{eqnarray*}
u(0, X_\tau) & = & u(\tau, X_0) - \int_0^\tau \frac{\partial u}{\partial t}(\tau-t, X_t) dt + \frac{1}{2} \int_0^\tau \Delta_{g(t)} u(t-\tau,X_t) dt + \mbox{martingale}\\
 & = & u(\tau, X_0) + \mbox{martingale,}
\end{eqnarray*}
so that
\[
E^x \left[ \varphi(X_\tau) \right] = E^x \left[ u(0, X_\tau) \right] = E^x \left[ u(\tau, X_0) \right] = u(\tau, x) = \int_M G(x, \tau, y, 0) \varphi(y) d vol_{g(0)}(y).
\]
Since $\varphi$ is arbitrary the claim is proved.
\end{proof}

\begin{lemma} \label{cutlocus}
$
\{
  t \in [0,T] 
  \, | \,
  X_t \in \Cut_{g(t)} (o)
\}
$
has Lebesgue measure zero almost surely.
\end{lemma}

\begin{proof}
Since by Lemma~\ref{density} for each $t \in (0,T]$ and any starting point $x \in M$, the law of $X_t$ under $P^x$ is absolutely continuous 
with respect to the $g(t)$-Riemannian volume measure, and since moreover the cut-locus $\Cut_{g(t)}(o)$ has $g(t)$-volume zero
(see e.g.~\cite[Theorem~7.253]{hackenbrochthalmaier} or \cite[Proposition~3.1]{chavel}), we have
\[
E^x \left[ \int_0^T 1_{\{X_t \in \Cut_{g(t)}(o)\}} dt \right] = \int_0^T P^x \left[ X_t \in \Cut_{g(t)} (o) \right] dt = 0,
\]
so that almost surely $\int_0^T 1_{\{X_t \in \Cut_{g(t)}(o)\}} dt = 0$.
\end{proof}


We now apply Lemma~\ref{ito2} to the process $\rho(t, X_t)$ 
up to singularity. 
As long as $X_t$ stays away from $o$ and the $g(t)$-cut-locus of $o$,
\begin{equation} \label{eq:Ito_preliminary}
d \rho(t, X_t) =  d \beta_t + \frac{1}{2} \left[ \Delta_{g(t)} \rho + 2 \frac{\partial \rho}{\partial t} \right] (t, X_t) dt , 
\end{equation}
where $\beta_t$ is the martingale term given by 
\[
\beta_t := \sum_{i = 1}^d \int_0^t H_i(s) \rhot(s,U_s) d W_s^i . 
\]
As we will observe in Lemma~\ref{negligible}, the singularity of $\rho(t,x)$ at $o$ is negligible.
The quadratic variation $\langle \beta \rangle_t$ of $\beta_t$ 
is computed as follows:
\begin{eqnarray*}
\langle \beta \rangle_t & = & \sum_{i = 1}^d \int_0^t \left[ H_i(s) \rhot(s,U_s) \right]^2 ds\\
& = & \sum_{i = 1}^d \int_0^t \left[ (U_s e_i) \rho(s,X_s) \right]^2 ds\\
& = & \int_0^t |\nabla_{g(s)} \rho(s,X_s)|^2 ds\\
& = & t.
\end{eqnarray*}
Thus $\beta_t$ is a standard one-dimensional Brownian motion. 
\begin{lemma}[Lemma~5 and Remark~6 in \cite{mccanntopping}]\label{mct}
The function $(t,x) \mapsto \rho(t,x)$ is smooth 
whenever $x \notin \{o\} \cup \Cut_{g(t)}(o)$, and
\[
\frac{\partial \rho}{\partial t}(t,x) 
= 
\frac{1}{2} \int_0^{\rho (t, x)} 
\frac{\partial g}{\partial t} (\dot{\gamma}(s), \dot{\gamma}(s)) 
ds,
\]
where $\gamma: [0, \rho (t,x)] \to M$ is the unique minimizing unit-speed $g(t)$-geodesic joining $o$ to $x$.
\end{lemma}

Let $\Cutst$ be the space-time cut-locus defined by
\[
\Cutst
:=
\{
(t,x,y) \in [0,T] \times M \times M
\, | \,
(x,y) \in \Cut_{g(t)}
\}.
\]
It is shown in \cite{mccanntopping} that $\Cutst$ is a closed subset in $[0,T] \times M \times M$.
Though they assumed $M$ to be compact, extension to the noncompact case is straightforward.
Since $[0,T] \times \{ o \} \times \{ o \}$ is a compact subset in $[0,T] \times M \times M$ and it is away from $\Cutst$,
we can take $r_1 >0$ so that 
\begin{equation}\label{r1}
d_{g(t)} ( o , \Cut_{g(t)} (o) ) > r_1
\end{equation}
holds for all $t \in [0,T]$. Thus we can use \eqref{eq:Ito_preliminary}) when $X_t$ is in a small neighbourhood of $o$ until $X_t$ hits $o$.
Since $g(t)$ is smooth, Lemma~\ref{mct} and \eqref{eq:Ito_preliminary} together with the Laplacian comparison theorem imply the following by a standard argument:

\begin{lemma} \label{negligible} 
With probability one, $X_t$ never hits $o$.
\end{lemma}

For $x,y \in M$, let 
\[
\doo(x,y) := \sup_{t \in [0,T]} d_{g(t)}(x,y).
\]
We consider $[0,T] \times M \times M$ equipped with the distance function $\dd((s,x_1,x_2), (t,y_1,y_2)) := \max\{ |t-s|, \doo(x_1,y_1), \doo(x_2,y_2) \}$.

By Lemma~\ref{mct} and the compactness of $M$, there exists a constant $C_1 >0$ such that 
\begin{equation} \label{eq:Lipschitz}
| d_{g(t)} (x,y) - d_{g(t')} (x,y) | \le C_1 |t - t'|
\end{equation} 
holds for any $t , t' \in [0,T]$ and $x,y \in M$. We now define a set $A$ by
\[
A 
:= 
\left\{ 
  ( t , x , y ) \in [0,T] \times M \times M 
  \, \left| \, 
  \begin{array}{l} 
    d_{g(t)} ( o , x ) \ge 2 i_M / 3 , 
    d_{g(t)} (o,y) = i_M / 3 
    \mbox{ and }
    \\
    d_{g(t)} (x,y) = d_{g(t)} (o,x) - d_{g(t)} (o,y)
  \end{array}
  \right.
\right\}.
\]
Note that $A$ is closed and hence compact since $d_{g(t)} (x,y)$ is continuous as a function of $(t,x,y)$.
Note that, for $(t,x,y) \in A$, $y$ is on a minimal $g(t)$-geodesic joining $o$ and $x$.
In particular, symmetry of the cutlocus implies that $A \cap \Cut_{ST} = \emptyset$. Thus we have 
\[
\delta_1 := \dd( A , \Cut_{ST} ) \wedge \frac{i_M}{3(C_1+1)} > 0.
\]
We define the function $\U: \R_+ \to \R_+$ by
\[
\U(r) := \frac{d-1}{2} K \coth \left( K \cdot r \wedge \frac{i_M}{3} \right) + 2 C_1.
\]
The Laplacian comparison theorem implies that, for all $( t , x , y ) \notin \Cutst$,
$
| ( \Delta_{g(t)} d_{g(t)} ( y , \cdot ) ) ( x ) | 
\leq 
(d-1) K \coth( K d_{g(t)} ( x , y ) ) 
$ 
and hence Lemma~\ref{mct} implies 
\begin{equation} \label{eq:dominant} 
\left|
    \frac{1}{2} ( \Delta_{g(t)} d_{g(t)} (y , \cdot ) ) (x)
    + 
    \frac{\partial }{\partial t} d_{g(t)} (y,x)
\right|
\leq 
\U ( d_{g(t)} ( x , y ) ).
\end{equation} 
\begin{lemma}\label{fund}
Let 
$(t_0, x_0) \in \Cut_{g(t_0)}(o)$ and $\delta \in (0, \delta_1 )$.
Let $X$ be a $g(t)$-Brownian motion starting at $x_0$ at time $t_0$. Let
$\tilde{T} := T \wedge ( t_0 + \delta ) \wedge \inf \{ t \geq t_0 \, | \, d_{g(t)} (x_0, X_t) = \delta \}$.
Then
\[
E \left[ \rho(t \wedge \tilde{T}, X_{t \wedge \tilde{T}}) - \rho(t_0, x_0) - \int_{t_0}^{t \wedge \tilde{T}} \U ( \rho ( s , X_s ) )  ds \right] \leq 0.
\]
\end{lemma}

\begin{proof}
We construct a point $\tilde{o} \in M$ as follows: we choose a minimizing unit-speed $g(t_0)$-geodesic $\gamma$ from $o$ to $x_0$ and define $\tilde{o} := \gamma(i_M/3)$.
Then by construction $( t_0 , x_0 , \tilde{o} ) \in A$. Moreover
for all $t \in [t_0, \tilde{T}]$ we have $\dd((t_0, x_0, \tilde{o}), (t,X_t,\tilde{o})) < \delta_1$ and therefore $X_t \notin \Cut_{g(t)}(\tilde{o})$. Let
\[
\rho^+(t,x) := d_{g(t)} ( o, \tilde{o} ) + d_{g(t)}(\tilde{o}, x). 
\]
Since $\tilde{o}$ lies on a minimizing $g(t_0)$-geodesic from $o$ to $x_0$, we have $\rho^+(t_0,x_0) = \rho(t_0,x_0)$.
Moreover, by the triangle inequality, $\rho^+(t,x) \geq \rho(t,x)$ for all $(t,x)$. On $[t_0, \tilde{T}]$,
\begin{eqnarray}
\nonumber
d_{g(t)}( \tilde{o} , X_t ) & \geq & d_{g(t)} ( \tilde{o} , x_0 ) - d_{g(t)} ( x_0 , X_t )\\
\nonumber & \geq & d_{g(t_0)} ( \tilde{o} , x_0 ) - (1 + C_1) \delta\\
\label{eq:distant} & \geq & \frac{i_M}{3}
\end{eqnarray}
holds. By \eqref{eq:dominant} and Lemma~\ref{mct}, 
\[
\left( 
  \frac12 \Delta_{g(t)} \rho^+ 
  + 
  \frac{\partial \rho^+}{\partial t}
\right) 
(t,x)
\le 
\U ( \rho^+ (t,x) ) 
\]
holds if $(t,x,\tilde{o}) \notin \Cutst$. Note that $\U ( \rho^+ (t,X_t) ) = \U ( \rho (t,X_t) )$ holds
for all $t \in [t_0, \tilde{T}]$ since we can show $\rho( t , X_t ) \geq i_M / 3$ in a similar way as in \eqref{eq:distant}. Therefore
\begin{align*}
\rho(t \wedge \tilde{T}, X_{t \wedge \tilde{T}}) 
& 
- \rho(t_0, X_{t_0}) 
- \int_{t_0}^{t \wedge \tilde{T}} \U( \rho ( s , X_s ) ) ds
\\
& =  
\rho(t \wedge \tilde{T}, X_{t \wedge \tilde{T}}) 
- \rho^+(t_0, X_{t_0}) 
- \int_{t_0}^{t \wedge \tilde{T}} \U ( \rho^+ ( s , X_s ) ) ds
\\
& \leq 
\rho^+(t \wedge \tilde{T}, X_{t \wedge \tilde{T}}) 
- \rho^+(t_0, X_{t_0})
- \int_{t_0}^{t \wedge \tilde{T}} 
\left( 
  \frac{1}{2} \Delta_{g(s)} \rho^+ 
  + 
  \frac{\partial \rho^+}{\partial s} 
\right)
(s,X_s) 
ds.
\end{align*}
Since $\rho^+$ is smooth at $(t,X_t)$ for $t \in [t_0, \tilde{T}]$, 
the last term is a martingale. 
Hence the claim follows.
\end{proof}

For $\delta \in (0 , \delta_1)$, 
we define a sequence of stopping times $(S_n^\delta)_{n \in \N}$ 
and $(T_n^\delta)_{n \in \N_0}$ by 
\begin{eqnarray*}
T_0^\delta 
& := & 
0,
\\
S_n^\delta 
& := & 
T \wedge 
\inf \{ 
  t \geq T_{n-1}^\delta 
  \, | \, 
  X_t \in \Cut_{g(t)}(o) 
\},
\\
T_n^\delta 
& := & 
T \wedge 
(S_n^\delta + \delta) \wedge 
\inf \{ 
  t \geq S_n^\delta 
  \, | \, 
  d_{g(t)} (X_{S_n^\delta}, X_t) = \delta
\}. 
\end{eqnarray*}
Note that these are well-defined 
because $\Cutst$ and 
$\{ (t,x) \, | \, d_{g(t)} ( y , x ) = \delta \}$, where $y \in M$, 
are closed. 
\begin{proposition}\label{sup}
The process
\[
\rho(t, X_t) - \rho(0, X_0) - \int_0^t \U(s,X_s) ds
\]
is a supermartingale.
\end{proposition}

\begin{corollary}
The process $\rho(t,X_t)$ is a semimartingale.
\end{corollary}

\begin{proof}[Proof of Proposition~\ref{sup}]
Thanks to the strong Markov property of Brownian motion it suffices to show that for all 
deterministic starting points $(t_0, x_0) \in [0,T] \times M$ and all $t \in [t_0, T]$
\[
E \left[ \rho(t, X_t) - \rho(t_0, X_{t_0}) - \int_{t_0}^t \U(\rho(s,X_s)) ds \right] \leq 0.
\]
To show this we first observe that thanks to Lemma~\ref{ito2}, \eqref{eq:dominant} and Lemma~\ref{fund} for all $n \in \N$
\[
E \left[ \rho(t \wedge S_n^\delta, X_{t \wedge S_n^\delta}) - \rho(t \wedge T_{n-1}^\delta, X_{t \wedge T_{n-1}^\delta})
- \int_{t \wedge T_{n-1}^\delta}^{t \wedge S_n^\delta} \U(\rho(s,X_s)) ds \, \Big| \, \F_{T_{n-1}^\delta} \right] \leq 0
\]
and
\[
E \left[ \rho(t \wedge T_n^\delta, X_{t \wedge T_n^\delta}) - \rho(t \wedge S_n^\delta, X_{t \wedge S_n^\delta})
- \int_{t \wedge S_n^\delta}^{t \wedge T_n^\delta} \U(\rho(s,X_s)) ds \, \Big| \, \F_{S_n^\delta} \right] \leq 0.
\]
It remains to show that $T_n \to T$ as $n \to \infty$.
If $\lim_{n \to \infty } T_n =: T_\infty < T$ occurs,
then $T_n^\delta - S_n^\delta$ converges to 0 as $n \to \infty$.
In addition,
$d_{g(t)} ( X_{S_n^\delta} , X_{T_n^\delta} ) = \delta$ must hold
for infinitely many $n \in \mathbb{N}$.
Take $N \in \mathbb{N}$ so large that $C_1 ( T_\infty -  T_n ) < \delta / 2$
for all $n \ge N$. Then \eqref{eq:Lipschitz} yields
$
d_{g(T_\infty)} ( X_{S_n^\delta} , X_{T_n^\delta} ) \ge \delta / 2 
$ 
for infinitely many $n \ge N$. 
But it contradicts with the fact 
that $X_t$ is uniformly continuous on $[0,T]$.  
Hence $T_n \to T$ as $n \to \infty$. 
\end{proof}

\begin{lemma} \label{null}
$
\lim_{\delta \to 0} 
\sum_{n=1}^\infty | T_n^\delta - S_n^\delta | 
= 0 
$
almost surely.  
\end{lemma}

\begin{proof}
For $\delta > 0$, 
let us define a random subset $E_\delta$ and $E$ 
in $[0, T]$ by  
\begin{align*}
E_\delta 
& : = 
\{ 
  t \in [ 0 , T ] 
  \, | \, 
  \mbox{there exists $t' \in [0,T]$ satisfying }
  |t - t'| \le \delta 
  \mbox{ and } 
  ( t' , X_{t'} , o ) \in \Cutst 
\} ,  
\\
E 
& : = 
\{ 
  t \in [ 0 , T ] 
  \, | \, 
  ( t , X_t , o ) \in \Cutst 
\} . 
\end{align*}  
Since 
the map $t \mapsto ( t , X_t , o )$ is continuous 
and 
$\Cutst$ is closed, 
$E$ is closed and hence $E = \cap_{ \delta > 0 } E_\delta$ holds. 
By the definition of $S_n^\delta$ and $T_n^\delta$, 
we have 
\[
E 
\subset 
\bigcup_{n=1}^\infty [ S_n^\delta , T_n^\delta ]  
\subset 
E_\delta
\] 
and hence the monotone convergence theorem implies  
\[
\lim_{\delta \to 0} 
\sum_{n=1}^\infty | T_n^\delta - S_n^\delta | 
\le   
\lim_{\delta \to 0} 
\int_0^T 1_{E_\delta} (t) dt 
= 
\int_0^T 1_E (t) dt    
= 0  
\]
almost surely, 
where the last equality follows 
from Corollary~\ref{cutlocus}. 
\end{proof}

\begin{lemma} \label{martingale}
The martingale part of $\rho(t,X_t)$ is 
\[
\sum_{i=1}^d \int_0^t (U_s e_i) \rho(s,X_s) d W_s^i.
\]
\end{lemma}

\begin{proof}
By the martingale representation theorem there exists an $\R^d$-valued process $\eta$ such that
the martingale part of $\rho(t,X_t)$ equals $\int_0^t \eta_s dW_s$. Let
\[
N_t := \int_0^t \eta_s dW_s - \sum_{i = 1}^d \int_0^t (U_s e_i) \rho(s,X_s) d W_s^i.
\]
Using the stopping times $S_n^\delta$ and $T_n^\delta$, the quadratic variation $\langle N \rangle_T$ of $N$ is expressed as follows; 
\begin{equation} \label{eq:quadratic}
\langle N \rangle_T 
= 
\sum_{i=1}^d 
\sum_{n=1}^\infty 
\left( 
    \int_{T_{n-1}^\delta \wedge T}^{S_n^\delta \wedge T} 
    | \eta_t^i - ( U_t e_i ) \rho ( t, X_t ) |^2 
    dt   
    + 
    \int_{S_n^\delta \wedge T}^{T_n^\delta \wedge T} 
    | \eta_t^i - ( U_t e_i ) \rho ( t, X_t ) |^2 
    dt 
\right) .
\end{equation}  
Since $X_t \notin \Cut_{g(t)} (o)$ if $t \in (T_{n-1}^\delta , S_n^\delta )$,  
It\^o's formula \eqref{eq:Ito_preliminary} yields 
\[
\int_{T_{n-1}^\delta \wedge T}^{S_n^\delta \wedge T} 
| \eta_t^i - ( U_t e_i ) \rho ( t, X_t ) |^2 
dt   
= 0 
\]
for $n \in \mathbb{N}$ and $i =1 , \cdots , d$. For the second term in the right-hand side of \eqref{eq:quadratic} we have
\[
\sum_{n=1}^\infty 
\int_{S_n^\delta \wedge T}^{T_n^\delta \wedge T}
| \eta_t^i - ( U_t e_i ) \rho ( t, X_t ) |^2
dt
\le
2 
\int_{\bigcup_{n=1}^\infty [ S_n^\delta , T_n^\delta ]}
\left(
    | \eta_t |^2
    +
    1
\right)
dt .
\]
Since $\eta_t$ is locally square-integrable on $[0,T]$ almost surely, Lemma~\ref{null} yields $\langle N \rangle_T =0$ and the conclusion follows.
\end{proof}

We can now conclude the proof of Theorem~\ref{itoradial}: 
Set 
$
I_\delta 
: = 
\bigcup_{n=1}^\infty 
[ S_n^\delta , T_n^\delta ]
$.
Set $L_t^\delta$ by 
\begin{align*}
L_t^\delta
& := 
- \rho(t,X_t) + \rho(0,X_0) 
 + 
\sum_{i=1}^d \int_0^t (U_s e_i) \rho(s,X_s) d W_s^i
\\
& \qquad 
+
\int_{[0,t] \setminus I_\delta}
\left[
    \frac{1}{2} \Delta_{g(t)} \rho 
    + 
    \frac{\partial \rho}{\partial s} 
\right] 
(s,X_s) 
ds
+  
\int_{[0,t] \cap I_\delta}
\U ( \rho ( s, X_s ) ) 
ds 
. 
\end{align*}
By Proposition~\ref{sup}, Lemma~\ref{martingale} and It\^o's formula~\eqref{eq:Ito_preliminary} on 
$[0,T] \setminus I_\delta$, 
$L^\delta_t$ is non-decreasing in $t$. 
In particular, 
$L^\delta_t$ can increase 
only when $t \in I_\delta$. 
Then we have 
\begin{align}
\nonumber
\rho(t, X_t) & - \rho(0, X_0) - \sum_{i = 1}^d \int_0^t (U_s e_i) \rho(s,X_s) d W_s^i
- \int_0^t \left[ \frac{1}{2} \Delta_{g(s)} \rho + \frac{\partial \rho}{\partial s} \right] (s, X_s) ds + L_t^\delta\\
\label{eq:error}
& = - \int_{[0,t] \cap I_\delta} \left[ \frac{1}{2} \Delta_{g(t)} \rho + \frac{\partial \rho}{\partial s} \right] (s,X_s) ds
- \int_{[0,t] \cap I_\delta} \U ( \rho ( s, X_s ) ) ds.
\end{align}
Since \eqref{eq:dominant} yields 
\begin{align*}
\left| 
    \int_{[0,t] \cap I_\delta} 
    \left[ 
        \frac{1}{2} \Delta_{g(t)} \rho 
        + 
        \frac{\partial \rho}{\partial s} 
    \right] 
    (s,X_s) 
    ds
    +  
    \int_{[0,t] \cap I_\delta}
    \U ( \rho ( s, X_s ) ) 
    ds 
\right| 
\le 
2 \int_{I_\delta} \U ( \rho ( s, X_s ) ) ds   
\end{align*}
and $\U ( \rho (s ,X_s) )$ is bounded on $I_\delta$, 
Lemma~\ref{null} yields that 
the right hand of \eqref{eq:error} converges to 0 
as $\delta \to 0$. 
Thus $L_t := \lim_{\delta \downarrow 0} L^\delta_t$ exists 
for all $t \in [0,T]$ almost surely 
and hence \eqref{eq:Itoradial} holds. 
We can easily deduce the fact that 
$L_t$ can increase only when $t \in \Cut_{g(t)} (o)$ 
from the corresponding property for $L^\delta_t$. \hfill $\Box$

\section{Proof of Theorem~\ref{mainresult}: Non-explosion of Brownian motion}

We define $k_1 \ge 1$ and $\bar{F} \: : \: [0,\infty) \to \R$ by
\begin{align*}
k_1 & := \inf \left\{ k \ge 1 \, | \, \Ric_{g(t)}(x) \geq -(d-1)k^2
  \mbox{ 
    for $t \in [0,T]$ 
    and $x \in M$ 
    with $\doo(o,x) \le r_1$}
\right\},
\\
\bar{F} (s) 
& := 
k_1 \coth ( k_1 \cdot s \wedge r_1 ) + k_1 \cdot s \wedge r_1,
\end{align*}
where $r_1$ is defined in \eqref{r1}.

Theorem~\ref{mainresult} follows immediately from the following estimate of the drift part of \eqref{eq:Itoradial}:

\begin{proposition} \label{bound}
Suppose \eqref{backwards}. 
Then, 
for all $(t,x) \in [0,T] \times M$ with $(t,x,o) \notin \Cut_{ST}$,
\[
\Delta_{g(t)} \rho (t,x) + 2 \frac{\partial \rho}{\partial t} (t,x) \leq \bar{F} ( \rho(t,x) ).
\]
\end{proposition}
To prove this proposition it suffices to show the following lemma: 
\begin{lemma} \label{bound_geodesic}
Suppose \eqref{backwards}. 
Fix $t \in [0,T]$ and a minimizing unit-speed $g(t)$-geodesic $\gamma: [0,b] \to M$ with $\gamma(0) = 0$.
Then there exists a non-increasing function $F \: : \: (0, b) \to \R$ satisfying $F(s) \le \bar{F} (s)$ and
\[
\Delta_{g(t)} \rho(t, \gamma (s)) + 2 \frac{\partial \rho}{\partial t} (t, \gamma(s)) \leq F(s)
\]
for all $s \in (0,b)$.
\end{lemma}

\begin{proof}
Let $( X_i )_{i=1}^d$ be orthonormal parallel fields along $\gamma$ with $X_1 = \dot{\gamma}$. Fix $r \in (0, b)$, and let 
$J_i$ be the Jacobi field along $\gamma |_{[0,r]}$ with $J_i(0) = 0$ and $J_i(r) = X_i(r)$. 
Then it is well known 
(see \cite{chavel} for example) 
that
\begin{align*}
( \Delta_{g(t)}  d_{g(t)} (\gamma (0), \cdot ) ) (\gamma(r))
& = \sum_{i=2}^d I(J_i, J_i),
\end{align*}
where the index form $I$ for smooth vector fields $Y,Z$ along $\gamma |_{[0,r]}$ is defined by 
\begin{align*}
I ( Y , Z )
& : =  
\int_0^r 
\left( 
  \langle \dot{Y}(s) , \dot{Z}(s) \rangle_{g(t)} 
  - 
  \langle 
    R_{g(t)} ( Y (s), \dot{\gamma}(s) ) \dot{\gamma}(s), Z(s)
  \rangle_{g(t)}
\right) 
ds.
\end{align*}
Let $G: [0,b] \to \R$ be the solution to the initial value problem 
\[
\begin{cases}
\displaystyle
G''(s) 
= 
- 
\frac{ 
  \Ric_{g(t)} (\dot{\gamma}(s) , \dot{\gamma}(s))
}
{d-1} 
G(s), 
\\
\displaystyle 
G(0) = 0, G'(0) = 1 . 
\end{cases}
\]
Then we have
\begin{align}
\nonumber
\sum_{i=2}^d I(G X_i, G X_i)
& =  
\sum_{i=2}^d 
\int_0^r 
\left[ 
  \left| G'(s) X_i(s) \right|^2 
  - 
  \langle 
    R( G(s) X_i(s), \dot{\gamma}(s) ) \dot{\gamma}(s), G(s) X_i(s) 
  \rangle 
\right]
ds
\\
\nonumber
& =
\int_0^r 
\left[ 
  (d-1) G'(s)^2
  - G(s)^2 \Ric(\dot{\gamma}(s), \dot{\gamma}(s))
\right] 
ds
\\
\nonumber
& =
(d-1) \int_0^r \left[ G'(s)^2 + G(s) G''(s) \right] ds
\\
\label{eq:index1}
& = 
(d-1) G(r) G'(r). 
\end{align} 
Since 
$\gamma (0)$ has no conjugate point along $\gamma$ 
on $[0,r]$, 
the left hand side of \eqref{eq:index1} must be strictly positive
(see Theorem~2.10 in \cite{chavel}). 
It follows that $G(r) > 0$ for all $r \in (0,b)$. 
Now let $Y_i(s) := \frac{G (s)}{G (r)} X_i(s)$. 
Note that $Y_i$ has the same boundary values as $J_i$.
Therefore, by the index lemma,
\begin{equation*}
( \Delta_{g(t)} d_{g(t)} ( \gamma (0), \cdot ) ) ( \gamma (r) )
\leq  
\sum_{i=2}^d I(Y_i, Y_i) 
= 
\frac{(d-1) G'(r)}{G(r)}. 
\end{equation*}
Hence Lemma~\ref{mct} and \eqref{backwards} yield 
\begin{eqnarray*}
(\Delta \rho + 2 \frac{\partial \rho}{\partial t})(t, \gamma(r)) & \leq & \frac{(d-1) G'(r)}{G(r)} + \int_0^r \Ric(\dot{\gamma}(s), \dot{\gamma}(s)) ds\\
& = & (d-1) \left[ \frac{G'(r)}{G(r)} - \int_0^r \frac{G''(s)}{G(s)} ds \right]\\
& =: & F(r).
\end{eqnarray*}
Since
\begin{equation*}
F'(r) 
= 
(d-1) \left[ \frac{G(r) G''(r) - G'(r)^2}{G(r)^2} - \frac{G''(r)}{G(r)} \right]
= 
- \frac{(d-1) G'(r)^2}{G(r)^2}  < 0 ,  
\end{equation*} 
$F$ is decreasing. 
In particular, we have $F (r) \le F ( r \wedge r_1 )$. 
A usual comparison argument implies 
$G' ( r \wedge r_1 ) / G ( r \wedge r_1 ) 
\le 
k_1 \coth ( k_1 \cdot r \wedge r_1 )$ 
and hence the conclusion follows. 
\end{proof}

\section{Generalization to non-symmetric diffusion}\label{nonsymmetric}

We generalise the previous results on more general setting 
including the case for non-symmetric diffusions. 
Let $X_t$ be time-dependent diffusion whose generator is 
$\Delta_{g(t)} / 2 + Z(t)$, 
where $Z(t)$ is a time-dependent vector field on $M$ 
which is smooth on $[0,T] \times M$. 

Even in this case, 
Theorem~\ref{itoradial} still holds by replacing 
$\Delta_{g(t)} / 2$ with $\Delta_{g(t)} / 2 + Z (t)$. 
In what follows, we briefly mention the proof. 
Except for Lemma~\ref{density} and Lemma~\ref{cutlocus}, 
an extension of each assertion is straightforward. 
For Lemma~\ref{cutlocus}, some difficulties 
come from the fact that the result corresponding 
to Lemma~\ref{density}, 
especially the existence of a fundamental solution, 
is not yet known at this moment 
for non-symmetric diffusions. 
But, for our purpose, it suffices to show the following: 
\begin{lemma}
Suppose that $M$ is compact. 
Then $P^x [ X_t \in \Cut_{g(t)} (o) ] =0$.  
\end{lemma} 
\begin{proof}
Let $\hat{Z}(t)$ be a differential 1-form 
corresponding to $Z(t)$ 
by duality with respect to $g(t)$. 
Let $M^Z_t$ be the martingale part of 
the stochastic line integral of 
$\hat{Z} (t)$ along $X_t$. 
Note that there is a constant $c > 0$ 
such that 
$\langle M^Z \rangle_t \le c t$ holds  
since $M$ is compact. 
Let us define a probability measure $\tilde{P}^x$ 
on the same probability space as $P^x$ 
by 
$
\tilde{P}^x [ A ] 
: = 
E^x 
[ 
  \exp (- M^Z - \langle M^Z \rangle_t / 2 ) 
  1_A 
]
$. 
By the Girsanov formula, the law of $X_t$ 
under $\tilde{P^x}$ coincides with 
that of $g(t)$-Brownian motion at time $t$. 
The Schwarz inequality yields 
\begin{align*} 
P^x [ X_t \in \Cut_{g(t)} (o) ] 
& \le 
E^x 
\left[ 
  \mathrm{e}^{-M^Z_t - \langle M^Z \rangle_t / 2} 
  1_{\{ X_t \in \Cut_{g(t)} (o) \} } 
\right]^{1/2}    
E^x 
\left[ 
  \mathrm{e}^{M^Z_t + \langle M^Z \rangle_t / 2} 
\right]^{1/2}
\\
& \le 
\tilde{P}^x 
\left[ 
  X_t \in \Cut_{g(t)} (o) 
\right]^{1/2} 
\mathrm{e}^{ct/2}. 
\end{align*}  
Hence the conclusion follows from Lemma~\ref{density}. 
\end{proof}
To state an extension of Theorem~\ref{mainresult}, 
define a tensor field $( \nabla Z(t) )^\flat$ by 
\[
( \nabla Z (t) )^\flat (X,Y) 
:= 
\frac12 
\left( 
    \langle 
      \nabla_X Z (t) , Y 
    \rangle_{g(t)} 
    + 
    \langle 
      \nabla_Y Z (t) , X 
    \rangle_{g(t)} 
\right) 
,
\]
where $\nabla$ is the Levi-Civita connection with respect to $g(t)$. 
\begin{assump} \label{non-symmetric}
There exists a locally bounded measurable function $b$ on $[0,\infty)$ so that 
\begin{enumerate}
\item
$
( \nabla Z (t,x) )^\flat +  \partial_t g (t,x) 
\le 
\Ric_{g(t)}(x) + b ( \rho ( t , x ) ) g(t,x)
$      
for all $t \in (0,T)$ and all $x \in M$.
%
\item 
The 1-dimensional diffusion process $y_t$ given by
$
d y_t 
= 
d \beta_t 
+ 
\left( 
    \bar{F} ( y_t ) 
    + 
    \int_0^{y_t} b ( s ) ds 
\right) 
dt
$ 
does not explode. (This is the case if and only if
\[
\int_1^\infty \exp \left[ -2 \int_1^y \bbb(z) dz \right] \left\{ \int_1^y \exp \left[ 2 \int_1^z \bbb(\xi) d \xi \right] dz \right\} dy = \infty,
\]
where $\bbb(y) := \bar{F}(y) + \int_0^{y} b(s) ds$,
see e.g.~\cite[Theorem~6.50]{hackenbrochthalmaier} or \cite[Theorem~VI.3.2]{ikedawatanabe}.)
\end{enumerate}
\end{assump}

Once we obtain the following, non-explosion of $X_t$ 
follows in the same way as above 
by the comparison argument.
\begin{lemma} 
Suppose that Assumption~\ref{non-symmetric} holds. 
Fix $t \in [0,T]$ and a minimizing unit-speed $g(t)$-geodesic $\gamma: [0,b] \to M$ with $\gamma(0) = 0$.
Then there exists a constant $C_Z > 0$ depending only on $\{ Z(t) \}_{t \in [0,T]}$ and $\gamma(0)$ such that 
\[
( ( \Delta_{g(t)} + Z(t) ) d_{g(t)} ( \gamma (0), \cdot ) ) ( \gamma (s) ) 
+ 
2 \frac{\partial }{\partial t} d_{g(t)} (\gamma(0), \gamma(s)) 
\leq 
C_Z + F(s) 
+ 
\int_0^s b ( d_{g(t)} ( \gamma (0) , \gamma (u) ) ) du 
\] 
for all $s \in (0,b)$, where $F$ is the same function as appeared in Lemma~\ref{bound_geodesic}.
\end{lemma}
\begin{proof}
By a direct calculation, 
$
( \nabla Z (t) )^\flat 
( \dot{\gamma}(s) , \dot{\gamma}(s) ) 
= 
\partial_s 
\langle Z(t) , \dot{\gamma} (s) \rangle_{g(t)} ( \gamma (s) ) 
$. 
Hence we obtain 
\begin{align*}
( Z(t) d_{g(t)} ( \gamma (0) , \cdot ) ) ( \gamma (r) ) 
& = 
\langle Z(t) , \dot{\gamma} (r) \rangle_{g(t)} ( \gamma (r) ) 
\\ 
& = 
\langle Z(t) , \dot{\gamma} (0) \rangle_{g(t)} ( \gamma (0) ) 
+ 
\int_0^t 
( \nabla Z (t) )^\flat ( \dot{\gamma} (s) , \dot{\gamma}(s) ) 
ds . 
\end{align*}
Then, by setting 
$
C_Z 
:= 
\sup_{s \in [0,T]} 
\sup 
\{ 
  \langle Z(s) , X \rangle_{g(s)} ( \gamma (0) ) 
  \, | \, 
  X \in T_{\gamma (0)} M , 
  \| X \|_{g(s)} = 1 
\}
$,
the conclusion follows in a similar way as we did in the proof of Lemma~\ref{bound_geodesic}.
\end{proof}

\vspace{1cm}
\begin{flushright}
Kazumasa Kuwada\\
\vspace{.3cm}
\small
Graduate School of Humanities and Sciences 
\\
\small
Ochanomizu University \\
\small
Tokyo 112-8610, Japan
\\
\vspace{.2cm}
\small
\textit{e-mail}: \texttt{kuwada@math.ocha.ac.jp}
\\
\vspace{1cm}
\normalsize
Robert Philipowski\\
\vspace{.3cm}
\small
Institut f\"{u}r Angewandte Mathematik
\\
\small
Universit\"{a}t Bonn \\
\small
Endenicher Allee 60, 53115 Bonn, Germany
\\
\vspace{.2cm}
\small
\textit{e-mail}: \texttt{philipowski@iam.uni-bonn.de}
\end{flushright}

\end{document}